\pdfoutput=1
\documentclass[headings=small, paper=a4, final]{scrartcl}

\usepackage[T1]{fontenc}
\usepackage[utf8]{inputenc}
\usepackage{amsmath, amsthm, amssymb,mathrsfs}
\usepackage{hyperref,url}   

\usepackage{booktabs}        
\usepackage{rotating}
\usepackage{tikz}\usetikzlibrary{decorations.markings}

\newcommand{\CC}{\mathbb{C}}
\newcommand{\FF}{\mathbb{F}}

\newcommand{\QQ}{\mathbb{Q}}
\newcommand{\ZZ}{\mathbb{Z}}

\newcommand{\fp}{\mathfrak{p}}
\newcommand{\fo}{\mathfrak{o}}
\newcommand{\wt}[1]{\widetilde{#1}}

\newcommand{\Paho}{\mathscr{P}}
\newcommand{\PahoHy}{\mathscr{K}}
\newcommand{\Pamo}{\mathscr{J}}
\newcommand{\PahoSi}{\mathscr{P}}
\newcommand{\PahoKl}{\mathscr{Q}}
\newcommand{\Iwahori}{\mathscr{B}}

\newcommand{\Borel}{B}
\newcommand{\Siegel}{P}
\newcommand{\Klingen}{Q}

\newcommand{\Ione}{\mathbf{1}} 
\DeclareMathOperator{\parres}{\mathbf{r}}
\newcommand{\quadcharFq}{\lambda_0}
\newcommand{\cuspGL}{\pi}
\newcommand{\St}{\mathrm{St}}

\DeclareMathOperator{\Hom}{Hom}
\DeclareMathOperator{\GSp}{GSp}\DeclareMathOperator{\Sp}{Sp}
\DeclareMathOperator{\GL}{GL}
\DeclareMathOperator{\Rep}{Rep}
\DeclareMathOperator{\diag}{diag}
\DeclareMathOperator{\simi}{sim}
\DeclareMathOperator{\Ind}{Ind}

\newtheorem{theorem}{Theorem}[section]        
\newtheorem{lemma}[theorem]{Lemma}
\newtheorem{proposition}[theorem]{Proposition}
\newtheorem{corollary}[theorem]{Corollary}

\theoremstyle{definition}
\newtheorem{definition}[theorem]{Definition}

\title{
\begin{large}Parahoric Restriction for GSp(4)\end{large}
}
\author{Mirko R\"osner}
\date{}

\begin{document}
\maketitle
\begin{abstract}Parahoric restriction is the parahoric analogue of Jacquet's functor.
Fix an arbitrary parahoric subgroup of the group $\GSp(4,F)$ of symplectic similitudes of genus two over a local number field $F/\QQ_p$.
We determine the parahoric restriction of the non-cuspidal irreducible smooth representations in terms of explicit character values.
\end{abstract}

\section{Introduction}
For a reductive connected group $\mathbf{G}$ over a local number field $F/\QQ_p$, with group of $F$-valued points $G=\mathbf{G}(F)$, fix a compact parahoric subgroup $\Paho\subseteq G$ with pro-unipotent radical $\Paho^+$. The \emph{parahoric restriction functor} between categories of admissible representations
\begin{equation*}\parres_{\Paho}:\Rep(G)\to\Rep(\Paho/\Paho^+),\qquad(\rho,V)\mapsto (\rho|_{\Paho},V^{\Paho^+}).\end{equation*}
is the parahoric analogue of Jacquet's functor of parabolic restriction.
It assigns to an admissible representation $(\rho,V)$ of $G$ the action of the Levi quotient $\Paho/\Paho^+$ on the space of invariants under $\Paho^+$. The functor is exact and factors over semisimplification, so it is sufficient to study irreducible admissible representations.

For the group $\mathbf{G}=\GSp(4)$ of symplectic similitudes of genus two, we determine the parahoric restriction of non-cuspidal irreducible admissible representations in terms of explicit character values for the finite Levi quotient $\Paho/\Paho^+$.
This has applications in the theory of Siegel modular forms of genus two, invariant under principal congruence subgroups of squarefree level.

\subsection{Main result}
For a proper parabolic subgroup of $G=\GSp(4,F)$ fix a cuspidal irreducible admissible complex linear representation $\sigma$ of its Levi quotient. Let $\rho$ be a subquotient of the normalized parabolic induction of $\sigma$ to $G$. Then $\rho$ is non-cuspidal and every non-cuspidal irreducible admissible representation of $G$ arises this way.

The conjugacy classes of parahoric subgroups in $G$ are represented by the standard parahoric subgroups $\Iwahori,\PahoSi,\PahoKl,\PahoHy,\Pamo$ described below. We determine the parahoric restriction of $\rho$ with respect to each of these standard parahoric subgroups. If $\sigma$ has positive depth, then the parahoric restriction of $\rho$ is zero, because the depths of $\sigma$ and $\rho$ coincide \cite[5.2]{Moy-Prasad1996}.

\begin{theorem}\label{thm:02:F_1_GSp4}
Suppose $\sigma$ has depth zero.
i) The parahoric restriction of $\rho$ with respect to the standard hyperspecial parahoric subgroup $\PahoHy$ is given by Table~\ref{tab:02:tab_F1_on_GSp4}.\footnote{Partial results in this case have been obtained before by Breeding \cite{Breeding}.}

ii) The parahoric restriction of $\rho$ with respect to the standard Iwahori $\Iwahori$, standard Siegel parahoric $\PahoSi$ and standard Klingen parahoric $\PahoKl$ is given by Table \ref{tab:02:non-special_par_res_GSp4}.

iii) The parahoric restriction of $\rho$ with respect to the standard paramodular group $\Pamo$ is given by Table \ref{tab:02:tab_r_x_paramod}.
\end{theorem}
The proof is in Section \ref{sec:proof_hyperspecial}, the tables are in Section \ref{sec:Tables}.
By character theory, we can determine invariants for every subgroup of the Levi quotients. For the full Levi we obtain parahori-spherical vectors.\footnote{This coincides with previous work by Roberts and Schmidt \cite[Table A.15]{Roberts-Schmidt}.}

\begin{corollary}\label{Cor:spherical_vectors}
The parahori-spherical irreducible admissible representations of $G$ are exactly the subquotients of the unramified principal series $\mu_1\times\mu_2\rtimes\mu_0$ for unramified characters $\mu_1,\mu_2,\mu_0$ of $F^\times$. The dimension of parahori-spherical vectors is given by Table~\ref{tab:02:paho-spherical_non-cuspidal}.
\end{corollary}
\begin{proof}

For non-cuspidal irreducible admissible representations, the dimension of parahori-spherical vectors equals the multiplicity of the trivial representation in the parahoric restriction. Cuspidal representations are never parahori-spherical \cite[4.7]{Borel_Iwahori_Invariants}, \cite[6.11]{Moy-Prasad1996}.
\end{proof}
\begin{table}
\caption{Dimension of parahori-spherical vectors for unramified characters $\mu_1$, $\mu_2$, $\mu_0$, $\xi$. \label{tab:02:paho-spherical_non-cuspidal}}
\begin{small}
\centering
\begin{tabular}{lcccccc}
\toprule
type &$\rho$ of $\GSp(4,F)$             & $\dim\rho^{\PahoHy}$ & $\dim\rho^{\Pamo}$ & $\dim\rho^{\PahoSi}$ & $\dim\rho^{\PahoKl}$ & $\dim\rho^{\Iwahori}$ \\\midrule
   I &$\mu_1\times\mu_2\rtimes\mu_0$          & $1$                  &$2$                 & $4$                  & $4$                  & $8$  \\ 
  IIa&$\mu_1\St\rtimes\mu_0$                  & $0$                  &$1$                 & $1$                  & $2$                  & $4$  \\
  IIb&$\mu_1\Ione\rtimes\mu_0$                & $1$                  &$1$                 & $3$                  & $2$                  & $4$  \\
 IIIa&$\mu_1\rtimes\mu_0\St$                  & $0$                  &$0$                 & $2$                  & $1$                  & $4$  \\
 IIIb&$\mu_1\rtimes\mu_0 \Ione$               & $1$                  &$2$                 & $2$                  & $3$                  & $4$  \\
  IVa&$\mu_0 \St_{\GSp(4,F)}$                 & $0$                  &$0$                 & $0$                  & $0$                  & $1$  \\
  IVb&$L(\nu^2,\nu^{-1}\mu_0 \St)$            & $0$                  &$0$                 & $2$                  & $1$                  & $3$  \\
  IVc&$L(\nu^{3/2} \St,\nu^{-3/2}\mu_0)$      & $0$                  &$1$                 & $1$                  & $2$                  & $3$  \\
  IVd&$\mu_0\Ione_{\GSp(4,F)}$                & $1$                  &$1$                 & $1$                  & $1$                  & $1$  \\
   Va&$\delta(\left[\xi_u,\nu\xi_u\right],\nu^{-1/2}\mu_0)$&$0$      &$0$                 & $0$                  & $1$                  & $2$  \\
   Vb&$L(\nu^{1/2}\xi_u \St,\nu^{-1/2}\mu_0)$ &$0$                   &$1$                 & $1$                  & $1$                  & $2$  \\
   Vc&$L(\nu^{1/2}\xi_u \St,\nu^{-1/2}\xi_u\mu_0)$& $0$              &$1$                 & $1$                  & $1$                  & $2$  \\
   Vd&$L(\nu\xi_u,\xi_u\rtimes\nu^{-1/2}\mu_0)$ &$1$                 &$0$                 & $2$                  & $1$                  & $2$  \\
  VIa&$\tau(S,\nu^{-1/2}\mu_0)$                &$0$                  &$0$                 & $1$                  & $1$                  & $3$  \\ 
  VIb&$\tau(T,\nu^{-1/2}\mu_0)$                &$0$                  &$0$                 & $1$                  & $0$                  & $1$  \\ 
  VIc&$L(\nu^{1/2} \St,\nu^{-1/2}\mu_0)$       &$0$                  &$1$                 & $0$                  & $1$                  & $1$  \\ 
  VId&$L(\nu,1_{F^{\times}}\rtimes\nu^{-1/2}\mu_0)$&$1$              &$1$                 & $2$                  & $2$                  & $3$  \\
\bottomrule
\end{tabular}
\end{small}
\end{table}

\section{Preliminaries}
Fix a nonarchimedean local number field $F$ with finite residue field $\fo/\fp\cong\FF_q$ of order $q$. The valuation character $\nu=\vert\cdot\vert$ of $F^\times$ is normalized such that $\vert\varpi\vert=q^{-1}$ for a uniformizing element $\varpi\in\fp$. 

\subsection{The parahoric restriction functor}\label{s:On_par_res_functor}
Let $G$ be the group of $F$-rational points of a connected reductive linear algebraic group over a non-archimedean local number field $F$.  Let $\Paho\subseteq G$ be a parahoric subgroup with Levi decomposition
\begin{equation}
 1\to \Paho^+\to \Paho\to\underline{\Paho}\to 1.
\end{equation} where $\underline{\Paho}$ is a finite reductive group over the residue field.
For an admissible complex representation $\pi:G\to \mathrm{Aut}(V)$ on a complex vector space $V$, the action of $\Paho$ preserves the subspace $V^{\Paho^+}$ of $\Paho^+$-invariants in $V$. This defines a representation $(\rho|_{\Paho},V^{\Paho^+})$ of $\Paho/\Paho^+\cong\underline{\Paho}$. 
An intertwining operator $V_1\to V_2$ between admissible representations $(\rho_1,V_1)$ and $(\rho_2,V_2)$ of $G$ defines a canonical $\underline{\Paho}$-intertwiner $V_1^{\Paho^+}\to V_2^{\Paho^+}$.
\begin{definition} The \emph{parahoric restriction functor} for $\Paho$ is the exact functor between categories of admissible complex linear representations
\begin{equation}
 \parres_{\Paho}:\Rep(G)\to\Rep(\underline{\Paho}),\quad\begin{cases}(\rho,V)\mapsto (\rho|_{\Paho},V^{\Paho^+})\,,\\(V_1\to V_2)\mapsto (V_1^{\Paho^+}\to V_2^{\Paho^+})\,.\end{cases}
\end{equation}
\end{definition}

For parahoric subgroups $\Paho_2\subseteq \Paho_1\subseteq G$, parahoric restriction is \emph{transitive} 
\begin{equation}\label{eq:trans_parahori_res}
 \parres_{\Paho_2}(\rho,V)  \qquad  \cong  \qquad  \parres_{\Paho_2/\Paho_1^+}\circ\; \parres_{\Paho_1} (\rho,V),
\end{equation}
where $\parres_{\Paho_2/\Paho_1^+}:\Rep(\underline{\Paho_1})\to\Rep(\underline{\Paho_2})$
is the parabolic restriction functor with respect to the parabolic subgroup $\Paho_2/\Paho_1^+\subseteq \Paho_1/\Paho_1^+\cong\underline{\Paho_1}$ and Levi quotient $\underline{\Paho_2}$. Compare Vign\'{e}ras \cite[4.1.3]{Vig_Schur}.

The \emph{depth} of an irreducible admissible representation $\rho$ of $G$ is defined in the sense of Moy and Prasad \cite{Moy-Prasad1996}. By definition, an irreducible smooth representation of $G$ has depth zero if and only if it admits non-zero parahoric restriction with respect to some parahoric subgroup.


Let $P$ be a parabolic subgroup of $G$ with Levi subgroup $M$ and let $S\subseteq M$ be a maximal $F$-split torus of $G$. Fix an admissible irreducible representation $\sigma$ of $M$ and an irreducible subquotient $\rho$ of its parabolic induction to $G$. Let $\Paho$ be a parahoric subgroup attached to a point in the apartment of $S$.

\begin{proposition}\label{Prop:F_1_nonzero_for_subquots_of_par_ind}
If $\sigma$ has non-zero parahoric restriction with respect to the parahoric subgroup $M\cap\Paho$ of $M$, then $\rho$ has non-zero parahoric restriction with respect to $\Paho$.
\end{proposition}
\begin{proof}
Replace $P$ by an associate parabolic with the same Levi subgroup, so that there is a monomorphism $\rho\hookrightarrow \Ind_{P}^G(\sigma)$ in $\Rep(G)$ \cite[2.5]{Moy-Prasad1996}.
By Frobenius reciprocity, there is an epimorphism $\parres_{{P}}(\rho)\twoheadrightarrow \sigma$ in $\Rep(M)$, where $\parres_{P}$ denotes Jacquet's functor from $\Rep(G)$ to $\Rep(M)$.
By exactness of the functor of $\Paho^+\cap M$-invariants there is an epimorphism $$(\parres_{P}(\rho))^{M\cap\Paho^+}\twoheadrightarrow \sigma^{M\cap\Paho^+}\,.$$ 
Since $M\cap\Paho^+$ is the pro-unipotent radical of $M\cap\Paho$, the right hand side is non-zero by assumption and therefore $(\parres_{P}(\rho))^{M\cap\Paho^+}$ is also non-zero.
Since $\Paho^+\subseteq G$ admits Iwahori decomposition with respect to $P$ and $M$ \cite[4.2]{Moy-Prasad1996}, there is a surjection of vector spaces $\rho^{\Paho^+}\twoheadrightarrow (\parres_{P'}(\rho))^{M\cap \Paho^+}$ \cite[2.2]{Moy-Prasad1996}. Especially, $\rho^{\Paho^+}$ is non-zero.
\end{proof}

\subsection{Parahoric restriction for \texorpdfstring{$\GL(1)$}{GL(1)} and \texorpdfstring{$\GL(2)$}{GL(2)}}\label{ss:examples_par_res_GL}
For $G=\GL(1,F)$, the irreducible admissible representations are the smooth characters $\mu:F^\times\to\CC^\times$. The unique parahoric subgroup is $\fo^\times$ with pro-unipotent radical $1+\fp$.
If $\mu$ is tamely ramified or unramified, its parahoric restriction $\parres_{\fo^\times}(\mu)$ is the character $\wt{\mu}:\fo^\times/(1+\fp)\to\CC$ such that \begin{equation}\label{eq:hy_res_GL(1)}\mu(x)=\wt{\mu}(x (1+\fp))\qquad\text{for}\qquad x\in\fo^\times\,.
\end{equation} If $\mu$ is wildly ramified, its parahoric restriction is $\wt{\mu}=0$.

For $G=\GL(2,F)$ fix the standard Borel $B$ of upper triangular matrices and the maximal torus $T$ of diagonal matrices. The conjugacy classes of parahoric subgroups are represented by the standard hyperspecial parahoric subgroup $\PahoHy=\PahoHy_G=\GL(2,\fo)$ with $\PahoHy_G^+=\PahoHy_{G}\cap(I_2+\left(\begin{smallmatrix}\fp&\fp\\\fp&\fp\end{smallmatrix}\right))$
and the standard Iwahori $\Iwahori_G=\PahoHy_G\cap\left(\begin{smallmatrix}\fo&\fo\\\fp&\fo\end{smallmatrix}\right)$ with $\Iwahori=\Iwahori_G^+=\Iwahori_G\cap(I_2+\left(\begin{smallmatrix}\fp&\fo\\\fp&\fp\end{smallmatrix}\right))$.
We identify $\PahoHy_G/\PahoHy_G^+ \cong\GL(2,\fo/\fp)\cong\GL(2,q)$ by the canonical isomorphism and identify 
$\Iwahori_G/\Iwahori_G^+\cong (\fo/\fp)^\times \times(\fo/\fp)^\times$ via $x\mapsto (x_{11}\fp,x_{22}\fp)$.

The irreducible smooth representations of $G$ are
\begin{enumerate}
\item the principal series $\mu_1\times\mu_2=\mathrm{Ind}_B^G(\mu_1\boxtimes\mu_2)$ with $\mu_1\mu_2^{-1}\neq\nu^{\pm1}$,
\item one-dimensional representations $\mu_1\Ione_G=\mu_1\circ\det$,
\item twists of the Steinberg representation $\mu_1\St_G=(\mu_1\circ\det)\otimes\St_G$,
\item cuspidal irreducible representations $\pi$
\end{enumerate} for smooth characters $\mu_1,\mu_2$ of $F^\times$.
\begin{lemma}\label{Lemma:hy_res_GL(2)}The parahoric restriction of the irreducible admissible representations $\rho$ of $G$ at $\PahoHy$ and $\Iwahori$ is given by Table~\ref{tab:F_1_for_GL2}.\footnote{This is well-known, compare Vign\'{e}ras \cite[III.3.14]{Vigneras_Representations_Modulaire} or Bushnell and Henniart \cite[\S14,\S15]{BushnellHenniart}.}
\end{lemma}

\begin{proof}
For a pair of smooth characters $\mu_1,\mu_2$ of $F^\times$, the parahoric restriction of $\mu_1\times\mu_2$ at $\PahoHy_G$ is
\begin{equation*}\parres_{\PahoHy_G}(\mu_1\times\mu_2) \cong \wt{\mu_1}\times\wt{\mu_2},\end{equation*}
by a standard argument using Iwasawa decomposition $G=B\PahoHy_G$, compare Prop.~\ref{Prop:F_1_commutes_par_ind}. By exactness of parahoric restriction and Maschke's theorem, the exact sequence
\begin{align*}
0\longrightarrow\mu_1\Ione_{\GL(2,F)}\longrightarrow\nu^{-1/2}\mu_1\times\nu^{1/2}\mu_1\longrightarrow\mu_1\St_{\GL(2,F)}\longrightarrow0.
\end{align*}
yields an isomorphism
\begin{align*}
\wt{\mu_1}\St_{\GL(2,q)}\,\oplus\;\wt{\mu_1}\Ione_{\GL(2,q)}
\cong\wt{\mu_1}\times\wt{\mu_1}
\cong\parres_{\PahoHy_G}(\mu_1\Ione_{G})\oplus\parres_{\PahoHy_G}(\mu_1 \St_{G}).
\end{align*} If $\mu_1$ is tamely ramified or unramified, we have $\parres_{\PahoHy_G}(\mu_1\circ\det)=\wt{\mu_1}\circ\det$ because $\det (\PahoHy_G^+)= 1+\fp$. Cuspidal irreducible representations of depth zero are compactly induced from the normalizer of $\PahoHy$ \cite[6.8]{Moy-Prasad1996}, the result is then implied by a theorem of Vign\'{e}ras \cite[Cor.\,5.3]{Vigneras_Barcelona}. 
For Iwahori restriction use transitivity \eqref{eq:trans_parahori_res}.
\end{proof}

\begin{table}
\centering
\caption{Parahoric restriction for smooth irreducible  representations of $\GL(2,F)$.\label{tab:F_1_for_GL2}}
\begin{tabular}{lcc}
\toprule
$\rho$ of $\GL(2,F)$& $\parres_{\PahoHy}(\rho)$ of $\GL(2,q)$ &  $\parres_{\Iwahori}(\rho)$ of $\GL(1,q)\times \GL(1,q)$   \\
\midrule 
$\mu_1\times\mu_2$  & $\wt{\mu_1}\times\wt{\mu_2}$ & $\wt{\mu_1}\boxtimes\wt{\mu_2}+\wt{\mu_2}\boxtimes\wt{\mu_1}$\\
$\mu_1 \cdot \Ione_{\GL(2,F)}$ & $\wt{\mu_1}\cdot\Ione_{\GL(2,q)}$ &  $\wt{\mu_1}\boxtimes\wt{\mu_1}$                   \\ 
$\mu_1 \cdot \St_{\GL(2,F)}$   & $\wt{\mu_1}\cdot\St_{\GL(2,q)}$   &  $\wt{\mu_1}\boxtimes\wt{\mu_1}$                   \\
$\pi$   depth zero                & cuspidal irreducible    &  $0$                                               \\ 
$\pi$   positive depth            & 0                       &  $0$                                               \\
\bottomrule
\end{tabular}
\end{table}

\subsection{The group \texorpdfstring{$\GSp(4)$}{GSp(4)}}
The group $\mathbf{G}=\GSp(4)$ of symplectic similitudes of genus two is defined over $\ZZ$ by the equation
\begin{gather*}
 J=\nu gJg^t  \qquad  \text{for} \qquad g\in \GL(4),\; \nu\in\GL(1) \qquad \text{and} \qquad J=\left(\begin{smallmatrix}&I_2\\-I_2&\end{smallmatrix}\right).
\end{gather*}
The similitude factor $\simi(g)=\nu$ is uniquely determined by $g$ and defines a character $\simi:\mathbf{G}\to\GL(1)$. We fix the split torus $\mathbf{T}$ of diagonal matrices and
the standard parabolic subgroups
\begin{equation*}
\mathbf{\Borel}=\left(\begin{smallmatrix}\ast&\ast&\ast&\ast\\&\ast&\ast&\ast\\&&\ast&\\&&\ast&\ast\end{smallmatrix}\right)\cap \mathbf{G},\qquad
\mathbf{\Siegel}=\left(\begin{smallmatrix}\ast&\ast&\ast&\ast\\\ast&\ast&\ast&\ast\\&&\ast&\ast\\&&\ast&\ast\end{smallmatrix}\right)\cap \mathbf{G},\qquad
\mathbf{\Klingen}=\left(\begin{smallmatrix}\ast&\ast&\ast&\ast\\&\ast&\ast&\ast\\&&\ast&\\&\ast&\ast&\ast\end{smallmatrix}\right)\cap \mathbf{G}.
\end{equation*} The corresponding groups of $F$-rational points are $G, T,\Borel,\Siegel,\Klingen$.

For smooth characters $\mu_i$ of $F^\times$, $i=0,1,2$, normalized parabolic induction of the character $T\to\CC,\;\diag(t_1,t_2,t_0/t_1,t_0/t_2)\mapsto\mu_1(t_1)\mu_2(t_2)\mu_0(t_0)$ via the standard Borel $B$ yields the admissible representation $\mu_1\times\mu_2\rtimes\mu_0$ of $G$. For parabolic induction via $\Siegel$ and $\Klingen$ the notation is analogous, compare Tadi\'{c} \cite{Ta91}.

A non-trivial additive character $\psi:F\to\CC^\times$ gives rise to a generic character of $U_\Borel$ via $\psi_U:U_\Borel\to\CC,\,u\mapsto \psi(u_{12}+u_{24})$. An admissible representation $\rho$ of $G$ is \emph{generic} if it admits a non-trivial $U_\Borel$-intertwining operator $(\rho|_{U_\Borel},V)\to(\psi_U,\CC)$.

We review the classification of standard parahoric subgroups of $\GSp(4,F)$.
The character group $X^\ast(\mathbf{T})=\Hom(\mathbf{T},\GL(1))$ is generated as a free group by the characters $e_i:\diag(t_1,t_2,t_0/t_1,t_0/t_2)\mapsto t_i$ for $i=0,1,2$. 
The simple affine roots $\psi_0=-(2e_1-e_0)+1,\psi_1=e_1-e_2$ and $\psi_2=2 e_2-e_0$ constitute the affine Dynkin diagram
\begin{center}
  \begin{tikzpicture}[scale=.5, decoration={markings,mark=at position 0.7 with {\arrow{>}}}]
    \draw (-2,0) node[anchor=east] {$\mathscr{C}_2:$};
    \draw[xshift=0.0 cm,thick] (0.0 cm,0) circle (.3cm);
    \draw[xshift=1.5 cm,thick] (1.5 cm,0) circle (.3cm);
    \draw[xshift=3.0 cm,thick] (3.0 cm,0) circle (.3cm);
    \draw[xshift=0.0 cm] (0.0 cm,-1) node[anchor=mid] {$\psi_0$};
    \draw[xshift=1.5 cm] (1.5 cm,-1) node[anchor=mid] {$\psi_1$};
    \draw[xshift=3.0 cm] (3.0 cm,-1) node[anchor=mid] {$\psi_2$};
    \draw[thick,double,postaction={decorate}](5.7 cm, 0.0 cm) -- +(-2.4 cm,0);
    \draw[thick,double,postaction={decorate}](0.3 cm, 0.0 cm) -- +(+2.4 cm,0);
    \draw (7.4,-.1) node[anchor=east] {$.$};
  \end{tikzpicture}
\end{center} Let $N(T)$ be the normalizer of $T$ in $G$. The affine Weyl group $N(T)/\mathbf{T}(\fo)$ is generated by the root reflections $s_i$ at $\psi_i$ for $i=0,1,2$ and the Atkin-Lehner element $u_1$
\begin{align*}
s_0=\left(\begin{smallmatrix}
     &&\varpi^{-1}&\\&1&&\\-\varpi&&&\\&&&1
    \end{smallmatrix}\right),\;\;
s_1=\left(\begin{smallmatrix}
     &1&&\\1&&&\\&&&1\\&&1&
    \end{smallmatrix}\right),\;\;
s_2=\left(\begin{smallmatrix}
     1&&&\\&&&1\\&&1&\\&-1&&
    \end{smallmatrix}\right),\;\;
u_1=\left(\begin{smallmatrix}
     &&&1\\&&1&\\&\varpi&&\\\varpi&&&
    \end{smallmatrix}\right).
\end{align*}
The simple affine roots $\psi_0$ and $\psi_2$ are conjugate under $u_1$.
The closed standard alcove $\mathcal{C}$ in the apartment attached to $T$ is defined by $\psi_i(x)\geq0$ for $i=0,1,2$. To each facet in $\mathcal{C}$ is attached one of the standard parahoric subgroups of $\GSp(4,F)$:
\begin{enumerate}
\item the standard Iwahori subgroup $\Iwahori$, attached to $\mathcal{C}$,
\begin{align*}
 \Iwahori=\simi^{-1}(\fo^\times)\cap \begin{pmatrix}\fo&\fo&\fo&\fo\\\fp&\fo&\fo&\fo\\\fp&\fp&\fo&\fp\\\fp&\fp&\fo&\fo\end{pmatrix},\quad \Iwahori^+=\Iwahori\cap\begin{pmatrix}1+\fp&\fo&\fo&\fo\\\fp&1+\fp&\fo&\fo\\\fp&\fp&1+\fp&\fp\\\fp&\fp&\fo&1+\fp\end{pmatrix}
\end{align*}
with Levi quotient $\Iwahori/\Iwahori^+\cong \GL(1,q)^3$ via $x\mapsto (x_{11},x_{22},\simi(x))$,
 \item the standard Siegel parahoric $\PahoSi$, attached to the facet $\psi_1^{-1}(0)\cap\mathcal{C}$, 
\begin{align*}
 \PahoSi=\simi^{-1}(\fo^\times)\cap \begin{pmatrix}\fo&\fo&\fo&\fo\\\fo&\fo&\fo&\fo\\\fp&\fp&\fo&\fo\\\fp&\fp&\fo&\fo\end{pmatrix},\quad \PahoSi^+=\PahoSi\cap \begin{pmatrix}1+\fp&\fp&\fo&\fo\\\fp&1+\fp&\fo&\fo\\\fp&\fp&1+\fp&\fp\\\fp&\fp&\fp&1+\fp\end{pmatrix}
\end{align*}
with $\PahoSi/\PahoSi^+\cong \GL(2,q)\times\GL(1,q)$ via $x\mapsto (\left(\begin{smallmatrix}x_{11}&x_{12}\\x_{21}&x_{22}\end{smallmatrix}\right),\simi(x))$,
 \item the standard Klingen parahoric $\PahoKl$, attached to the facet $\psi_2^{-1}(0)\cap \mathcal{C}$, 
\begin{align*}
  \PahoKl=\simi^{-1}(\fo^\times)\cap \begin{pmatrix}\fo&\fo&\fo&\fo\\\fp&\fo&\fo&\fo\\\fp&\fp&\fo&\fp\\\fp&\fo&\fo&\fo\end{pmatrix},\quad
  \PahoKl^+=\PahoKl\cap \begin{pmatrix}1+\fp&\fo&\fo&\fo\\\fp&1+\fp&\fo&\fp\\\fp&\fp&1+\fp&\fp\\\fp&\fp&\fo&1+\fp\end{pmatrix}
\end{align*}
with $\PahoKl/\PahoKl^+\cong \GL(1,q)\times\GSp(2,q)$ via $x\mapsto (x_{11},\left(\begin{smallmatrix}x_{22}&x_{24}\\x_{42}&x_{44}\end{smallmatrix}\right))$,
\item the standard hyperspecial parahoric subgroup $\PahoHy=\GSp(4,\fo)$, attached to the facet $\psi_1^{-1}(0)\cap\psi_2^{-1}(0)\cap \mathcal{C}$, 
\begin{align*}
 \PahoHy=\simi^{-1}(\fo^\times)\cap \begin{pmatrix}\fo&\fo&\fo&\fo\\\fo&\fo&\fo&\fo\\\fo&\fo&\fo&\fo\\\fo&\fo&\fo&\fo\end{pmatrix},\quad \PahoHy^+=\PahoHy\cap\begin{pmatrix}1+\fp&\fp&\fp&\fp\\\fp&1+\fp&\fp&\fp\\\fp&\fp&1+\fp&\fp\\\fp&\fp&\fp&1+\fp\end{pmatrix}
\end{align*}
with the canonical map $\PahoHy/\PahoHy^+\cong\GSp(4,\fo/\fp)\cong\GSp(4,q)$,
 \item the standard paramodular subgroup $\Pamo$ with facet $\psi_0^{-1}(0)\cap\psi_2^{-1}(0)\cap \mathcal{C}$,
\begin{gather*}\!\!\!\!\!\!
\Pamo=\simi^{-1}(\fo^\times)\cap \begin{pmatrix}\fo&\fo&\fp^{-1}&\fo\\\fp&\fo&\fo&\fo\\\fp&\fp&\fo&\fp\\\fp&\fo&\fo&\fo\end{pmatrix},\quad
\Pamo^+=\Pamo\cap \begin{pmatrix}1+\fp&\fo&\fo&\fo\\\fp&1+\fp&\fo&\fp\\\fp^2&\fp&1+\fp&\fp\\\fp&\fp&\fo&1+\fp\end{pmatrix}
\end{gather*}
with $\Pamo/\Pamo^+\cong(\GL(2,q)^2)^0:=\{(a,b)\in \GL(2,q)^2\,\vert\,\det a=\det b\}$ via 
\begin{equation*}x\mapsto \left(\begin{pmatrix}x_{11}&x_{13}\varpi\\x_{31}\varpi^{-1}&x_{33}\end{pmatrix},\begin{pmatrix}x_{22}&x_{24}\\x_{42}&x_{44}\end{pmatrix}\right)\,,\end{equation*}
\item the parahoric $u_1^{-1}\PahoKl u_1$ attached to the facet $\psi_0^{-1}(0)\cap \mathcal{C}$, 
\item the hyperspecial parahoric $u_1^{-1}\PahoHy u_1$ attached to $\psi_0^{-1}(0)\cap\psi_1^{-1}(0)\cap\mathcal{C}$.
\end{enumerate}
Conjugation by the Atkin-Lehner element $u_1$ preserves $\Iwahori$, $\PahoSi$ and $\Pamo$. The standard maximal parahorics are $\PahoHy$, $\Pamo$ and $u_1^{-1}\PahoHy u_1$.

There are double coset decompositions
\begin{gather}\label{eq:double_cosets_GSp4}
\GSp(4,F)= B \PahoHy
= \Siegel\Pamo=\Klingen \Pamo\sqcup \Klingen s_1\Pamo
\end{gather} for the standard parabolics $\Borel$, $\Siegel$, $\Klingen$, where $\sqcup$ denotes the disjoint union. The proof is elementary and follows from Iwasawa and Bruhat decomposition.

For every parahoric subgroup of $\GSp(4,F)$, the image of the pro-unipotent radical under the similitude character is $1+\fp$. Therefore twisting a representation $\rho$ by a tamely ramified or unramified character $\mu$ of $F^\times$ commutes with parahoric restriction in the following sense:
\begin{gather}\label{eq:paraho_res_twist}
\parres_{\PahoHy}((\mu\circ\simi)\otimes\rho)\cong (\wt{\mu}\circ\simi)\otimes \parres_{\PahoHy}(\rho)\,,\\
\label{eq:paraho_res_twist_Pamo}
\parres_{\Pamo}((\mu\circ\simi)\otimes\rho)\cong (\wt{\mu}\circ\det)\otimes\parres_{\Pamo}(\rho)\,.
\end{gather}
with $\det:(\GL(2,q)^2)^0\to\FF_q^\times\,,\;(a,b)\mapsto \det a$.


\section{Proof of Theorem \ref{thm:02:F_1_GSp4}}\label{sec:proof_hyperspecial}
\begin{proposition}\label{Prop:paramod_res_Klingen_induced}
For an admissible representation $(\sigma,V_\sigma)$ of $\GSp(2,F)$ and a character $\mu_1:F^\times\to\CC^\times$, the parahoric restriction at $\Pamo$ of the Klingen induced representation $\mu_1\rtimes\sigma$ is
\begin{gather}
 \parres_{\Pamo}(\mu_1\rtimes\sigma)\cong[\wt{\mu_1}\times 1,\wt{\sigma}]\oplus[\wt{\sigma},\wt{\mu_1}\times 1]
\end{gather}
for $\wt{\mu_1}=\parres_{\fo^\times}(\mu_1)$ and $\wt{\sigma}=\parres_{\GL(2,\fo)}(\sigma)$.
\end{proposition}
\begin{proof}
An explicit model of $\parres_{\Pamo}(\mu_1\rtimes\sigma)$ is given by the right-action of $\Pamo$ on 
\begin{gather*}
\wt{V}= \{f:G\to V_\sigma \,|\, f(pgk)=\delta_{\Klingen }^{1/2}(p)(\mu_1\boxtimes\sigma)(p) f(g)\;\forall p\in \Klingen ,g\in G,k\in\Pamo^+\}.
\end{gather*}
By \eqref{eq:double_cosets_GSp4}, any $f\in\wt{V}$ is uniquely determined by its restriction to $\Pamo$ and $s_1\Pamo$, so the $\Pamo$-representation $\wt{V}$ is isomorphic to the direct sum
\begin{gather*}
\{f|_{\Pamo}\,:\Pamo\to V_\sigma\,|\, f\in\wt{V} \}\;\;\oplus\;\; \{f|_{s_1\Pamo}\,:s_1\Pamo\to V_\sigma \,|\, f\in\wt{V}\}.
\end{gather*}
Every $f|_{\Pamo}$ in the first subspace is left invariant under $\Pamo^+\cap Q$, so it maps to the $\sigma$-invariants under $\GSp(2,\fo)^+$. Thus $f|_\Pamo$ factors over a unique function
\begin{gather*}
\wt{f}: \Pamo/\Pamo^+ \to V_\sigma^{\GSp(2,\fo)^+} \quad \text{with}  \quad  \wt{f}(qg)=\wt\mu(q_{11}) \wt\sigma\begin{pmatrix}q_{22}&q_{24}\\q_{42}&q_{44}\end{pmatrix}\wt{f}(g)
\end{gather*}
for every $g\in \Pamo/\Pamo^+\cong (\GL(2,q)^2)^0$ and every
\begin{gather*}q\in(\Pamo\cap \Klingen )\Pamo^+/\Pamo^+ \cong (\left(\begin{smallmatrix}\ast&\ast\\&\ast\end{smallmatrix}\right)\times\left(\begin{smallmatrix}\ast&\ast\\\ast&\ast\end{smallmatrix}\right))\cap(\GL(2,q)^2)^0.
\end{gather*}
By definition of the isomorphism $\Pamo/\Pamo^+\cong(\GL(2,q)^2)^0$, the space of these $\wt{f}$ is the induced representation $[\wt{\mu}\times 1,\wt{\sigma}]$.

For the second subspace the argument is analogous. 
\end{proof}
\begin{proposition}\label{Prop:paramod_res_Siegel_induced}
For smooth characters $\mu_0, \mu_1,\mu_2$ of $F^\times$ and an irreducible admissible representation $(\sigma,V_\sigma)$ of $\GL(2,F)$, the parahoric restriction of the Siegel induced representation $\sigma\rtimes\mu_0$ at the standard paramodular subgroup $\Pamo$ is
\begin{equation*}
\parres_{\Pamo}(\sigma\rtimes\mu_0) \quad\cong\quad
\begin{cases}  
\wt{\mu_0}[1\times\wt{\mu_1}, 1\times\wt{\mu_2}]
+\wt{\mu_0}[1\times\wt{\mu_2}, 1\times\wt{\mu_1}]& \sigma\cong\mu_1\times\mu_2,\\
\wt{\mu_0}[1\times\wt{\mu_1}, 1\times\wt{\mu_1}] & \sigma=\mu_1\St,\;\mu_1\Ione,\\
0                                                & \sigma \text{ cuspidal.}
\end{cases}\label{eq:r02_Siegel_St_ind}
\end{equation*}
\end{proposition}
\begin{proof}
By \eqref{eq:paraho_res_twist_Pamo} we can assume without loss of generality that $\mu_0=1$.  An explicit model $\wt{V}$ of $\parres_{\Pamo}(\sigma\rtimes1)$ is given by right-multiplication with elements of $\Pamo$ on the vector space of smooth functions $f:G\to V_{\sigma}$ with
\begin{gather*}
f(pgk)=\delta_{\Siegel}^{1/2}(p)\cdot\sigma\begin{pmatrix}p_{11}&p_{12}\\p_{21}&p_{22}\end{pmatrix} f(g)
\end{gather*}
for $p\in \Siegel $, $g\in G$ and $k\in\Pamo^+$. By the decomposition $G=\Siegel\Pamo$ \eqref{eq:double_cosets_GSp4}, every such $f$ is uniquely determined by its restriction to $\Pamo$. Therefore $\wt{V}$ is isomorphic to the vector space of $\Pamo^+$-invariant functions $$\wt{f}:\Pamo\to V_{\sigma}$$ which satisfy the condition
\begin{equation*}
 \wt{f}(pg)=\sigma\begin{pmatrix}p_{11}&p_{12}\\p_{21}&p_{22}\end{pmatrix} \wt{f}(g) \qquad\forall g\in\Pamo,\quad\forall p\in \Siegel \cap \Pamo.
\end{equation*}
Since $\wt{f}$ is left-invariant under every $p\in\Siegel \cap\Pamo^+$, the value $\wt{f}(g)\in V_\sigma$ is invariant under $\left(\begin{smallmatrix}1+\fp&\fo\\\fp&1+\fp\end{smallmatrix}\right)\subseteq\GL(2,F)$.
That means $\wt{f}(g)$ must be contained in the parahoric restriction $\parres_{\Iwahori_{\GL(2)}}(\sigma)$ with respect to the standard Iwahori $\Iwahori_{\GL(2)}\subseteq\GL(2,F)$.
By Lemma \ref{Lemma:hy_res_GL(2)}, $\parres_{\Iwahori_{\GL(2)}}(\sigma)$ is zero for cuspidal $\sigma$. For $\sigma=\mu\Ione,\mu\St$ it is isomorphic to $\wt{\mu_1}\boxtimes\wt{\mu_1}$ and the condition on $\wt{f}$ is
\begin{gather*}
 \wt{f}(pg)= \wt{\mu_1}(p_{11})\,\wt{\mu_1}(p_{22}) \wt{f}(g),\qquad \forall g\in \Pamo/\Pamo^+,\quad \forall p\in (\Siegel \cap \Pamo)/(\Siegel \cap \Pamo^+).
\end{gather*}
By construction of the isomorphism $\Pamo/\Pamo^+\cong(\GL(2,q)^2)^0$, the action of $\Pamo$ on $\wt{V}$ is the induced representation $[1\times\wt{\mu_1},1\times\wt{\mu_1}]$.
For the principal series $\sigma=\mu_1\times\mu_2$, the argument is analogous.
\end{proof}

\begin{proposition}\label{Prop:F_1_commutes_par_ind}
For admissible representations of $\GSp(4,F)$ that are parabolically induced, the parahoric restriction at $\PahoHy$ is given by
\begin{gather*}\label{eq:F_1_commutes_par_ind}
\parres_{\PahoHy}(\mu_1\times\mu_2\rtimes\mu_0)\cong \wt{\mu_1}\times\wt{\mu_2}\rtimes\wt{\mu_0},\\
\parres_{\PahoHy}(\mu_1\rtimes\sigma)\cong \wt{\mu_1}\rtimes\parres_{\GSp(2,\fo)}(\sigma),\\
\parres_{\PahoHy}(\sigma\rtimes\mu_0)\cong \parres_{\GL(2,\fo)}(\sigma)\rtimes\wt{\mu_0},
\end{gather*} for smooth characters $\mu_i:F^\times\to\CC^\times$ and an admissible representation $\sigma$ of $\GL(2,F)$.
\end{proposition}
\begin{proof}
The proof is similar to the previous propositions using $G=B\PahoHy$.
\end{proof}
\begin{proof}[Thm.\,\ref{thm:02:F_1_GSp4} i)] We only discuss the case of odd $q$; for even $q$ the proof is analogous. 
Irreducible representations $\rho$ of type I, II, III, VII and X are parabolically induced, so the result is clear by Prop.\,\ref{Prop:F_1_commutes_par_ind}, by \eqref{eq:hy_res_GL(1)} and Table~\ref{tab:F_1_for_GL2}.
Otherwise, $\rho$ is a non-trivial subquotient of a parabolically induced representation $\kappa$ and $\parres_{\PahoHy}(\kappa)$ is given by Prop.~\ref{Prop:F_1_commutes_par_ind}. If $\parres_\PahoHy(\kappa)=0$, then $\parres_\PahoHy(\rho)=0$ by exactness, otherwise $\parres_\PahoHy(\rho)$ is a non-zero subquotient of $\parres_\PahoHy(\kappa)$ by Prop.~\ref{Prop:F_1_nonzero_for_subquots_of_par_ind}.
It remains to determine the correct constituents of $\parres_\PahoHy(\kappa)$ case by case.
By \eqref{eq:paraho_res_twist} we can assume without loss of generality that $\mu_0=1$.

For the trivial representation $\rho=\Ione_{\GSp(4,F)}$ (type IVd), the hyperspecial parahoric restriction is trivial $\parres_\PahoHy(\rho)=\theta_0$. By \cite[(2.9)]{Roberts-Schmidt} and character theory \cite{Shinoda}
\begin{gather*}
\parres_\PahoHy(L(\nu^{2},\nu^{-1}\St_{\GSp(2,F)}))+\parres_\PahoHy(\Ione_{\GSp(4,F)})\cong\Ione_{\GL(2,q)}\rtimes1=\chi_3(1,1)\end{gather*}
decomposes as $\chi_3(1,1)=\theta_0+\theta_1+\theta_3$, so for type IVb 
\begin{gather*}\parres_\PahoHy(L(\nu^{2},\nu^{-1}\St_{\GSp(2,F)}))=\theta_1+\theta_3.\end{gather*} By the same argument we determine the parahoric restriction for type IVa and IVc as constituents of $\chi_1(1,1)=1\rtimes\Ione_{\GSp(2,q)}$ and $\chi_2(1,1)=1\rtimes\St_{\GSp(2,q)}$.

The representation $\rho=L(\nu^{1/2}\xi\St,\nu^{-1/2})$ of type Vb is a constituent of both \begin{equation*}\nu^{1/2}\xi\St_{\GL(2,F)}\rtimes\nu^{-1/2} \quad\text{and}\quad \nu^{1/2}\xi\Ione_{\GL(2,F)}\rtimes\nu^{-1/2}\xi\end{equation*} \cite[(2.10)]{Roberts-Schmidt}. Therefore the parahoric restriction $\parres_\PahoHy(\rho)$ must be a (non-zero) constituent of both $\wt{\xi}\St_{\GL(2,q)}\rtimes1$ and $\wt{\xi}\St_{\GL(2,q)}\rtimes\wt{\xi}$. By \cite{Shinoda}, the only common constituent is $\theta_1$ for unramified $\xi$ and $\tau_2$ for tamely ramified $\xi$. By exactness, types Va, Vc and Vd are clear.

The Klingen induced representation $1\rtimes\St_{\GL(2,F)}$ splits into the direct sum of $\tau(S,\nu^{-1/2})$ of type VIa and $\tau(T,\nu^{-1/2})$ of type VIb \cite[(2.11)]{Roberts-Schmidt}.
Its parahoric restriction is $\parres_{\PahoHy}(1\rtimes\St_{\GL(2,F)})=\theta_1+\theta_3+\theta_5$. The representation $\tau(S,\nu^{-1/2})$ is contained in $\nu^{1/2}\St\rtimes\nu^{-1/2}$ with restriction at $\PahoHy$ given by $\St_{\GL(2,q)}\rtimes1=\theta_1+\theta_4+\theta_5$; while $\tau(T,\nu^{-1/2})$ is contained in $\nu^{1/2}\Ione\rtimes\nu^{-1/2}$ with restriction $\theta_0+\theta_1+\theta_3$.
Therefore the pair of parahoric restrictions $(\parres_{\PahoHy}(\tau(S,\nu^{-1/2})),\parres_\PahoHy(\tau(T,\nu^{-1/2})))$ is either $(\theta_5+\theta_1,\theta_3)$ or $(\theta_5,\theta_1+\theta_3)$.
But the virtual representation
\begin{equation*}\tau(S,\nu^{-1/2})-\tau(T,\nu^{-1/2})\end{equation*}
is the endoscopic lift of $(\St_{\GL(2,F)},\St_{\GL(2,F)})$ in the sense of \cite{Weissauer200903}, so the trace of its parahoric restriction at $\PahoHy$ is zero on the $\GSp(4,q)$-conjugacy class stably conjugate to $\diag(a^q,a,a^{q^3},a^{q^2})$ \cite[Cor.\,4.25]{My_PhD_Thesis} for $a\in\FF^\times_{q^4}$ with $a^{q^2+1}\in\FF_q^\times$ and $a^{q-1}\neq\pm1$.
This implies $\parres_\PahoHy(\tau(S,\nu^{-1/2}))=\theta_5+\theta_1$ and $\parres_{\PahoHy}(\tau(T,\nu^{-1/2}))=\theta_3$. Types VIc and VId are clear by exactness.

For an irreducible cuspidal admissible representation $\cuspGL$ of $\GL(2,F)$ of depth zero, the Klingen induced representation $1\rtimes\cuspGL$ is a direct sum of two irreducible constituents, the generic $\tau(S,\cuspGL)$ of type VIIIa and the non-generic $\tau(T,\cuspGL)$ of type VIIIb.
The parahoric restriction $\parres_{\PahoHy}(1\rtimes\cuspGL)\cong1\rtimes\wt{\cuspGL}=X_3(\Lambda,1)=\chi_7(\Lambda)+\chi_8(\Lambda)$ has two irreducible constituents, so Prop.~\ref{Prop:F_1_nonzero_for_subquots_of_par_ind} implies that $\parres_\PahoHy(\tau(S,\cuspGL))$ is isomorphic to one of them and $\parres_\PahoHy(\tau(T,\cuspGL))$ is isomorphic to the other.
By a suitable character twist \eqref{eq:paraho_res_twist} we can assume that $\cuspGL$ is unitary. Then the virtual representation $\tau(S,\pi)-\tau(T,\pi)$ is the local endoscopic character lift of the representation $(\pi,\pi)$ of $\GL(2,F)^2/\GL(1,F)$ (antidiagonally embedded), compare \cite[Thm.\,4.5]{Weissauer200903}. For $\alpha,\beta\in\FF_{q^2}^\times$ with $\alpha,\beta,\alpha\beta,\alpha\beta^q\notin\FF_q^\times$, the trace of $\parres_\PahoHy(\tau(S,\pi)-\tau(T,\pi))$ on the stable conjugacy class with eigenvalues $\alpha\beta,\alpha\beta^q,\alpha^q\beta,\alpha^q\beta^q$ is
\begin{equation*}
 2(\Lambda(\alpha)+\Lambda(\alpha^q))(\Lambda(\beta)+\Lambda(\beta^q))\,,
\end{equation*}
see \cite[Cor.\,4.20]{My_PhD_Thesis}.
By Shinoda's character table \cite{Shinoda}, this coincides with the character value of $\chi_8(\Lambda)-\chi_7(\Lambda)$, but not with $\chi_7(\Lambda)-\chi_8(\Lambda)$. This implies $\parres_\PahoHy(\tau(S,\cuspGL))\cong\chi_8(\Lambda)$ and $\parres_\PahoHy(\tau(T,\cuspGL))\cong\chi_7(\Lambda)$.

For type IX, see \cite[\S3.2.1]{My_PhD_Thesis}.

Let $\rho=\delta(\nu^{1/2}\cuspGL,\nu^{-1/2})$ be an irreducible representation of type XIa where $\cuspGL$ is a cuspidal irreducible admissible representation of $\GL(2,F)$ with trivial central character. Then $\parres_{\PahoHy}(\rho)$ must be one of the two irreducible subquotients of
\begin{gather*}
 \parres_\PahoHy(\nu^{1/2}\cuspGL\rtimes\nu^{-1/2})=\wt{\pi}\rtimes1=X_2(\Lambda,1)=\chi_5(\omega_\Lambda,1)+\chi_6(\omega_\Lambda,1).
\end{gather*}
By \cite[Table A.12]{Roberts-Schmidt}, $\rho$ has paramodular level $\geq3$ and therefore does not admit non-zero invariants under the second paramodular congruence subgroup, which is conjugate to
\begin{gather*}
   \begin{pmatrix}\fo&\fp&\fo&\fp\\\fp&\fo&\fp&\fo\\\fo&\fp&\fo&\fp\\\fp&\fo&\fp&\fo\end{pmatrix}\cap \PahoHy.
\end{gather*} By character theory, $\chi_5(\omega_\Lambda,1)$ admits non-zero invariants under the image of this group in $\PahoHy/\PahoHy^+$,
so $\parres_\PahoHy(\rho)$ cannot be $\chi_5(\omega_\Lambda,1)$. The rest follows from exactness and Prop.~\ref{Prop:F_1_nonzero_for_subquots_of_par_ind}.
\end{proof}


\begin{proof}[Thm.\,\ref{thm:02:F_1_GSp4} ii)]
By transitivity of parahoric restriction \eqref{eq:trans_parahori_res}, this is implied by Thm.~\ref{thm:02:F_1_GSp4}. Parabolic restriction for $\GSp(4,q)$ can be determined explicitly by character theory. 
\end{proof}

For the paramodular subgroup $\Pamo\subseteq\GSp(4,F)$, the Atkin Lehner involution provides a symmetry condition:
\begin{lemma}\label{Lem:paramod_res_symmetry}
Let $\rho$ be an irreducible admissible representation of $\GSp(4,F)$. The parahoric restriction $\parres_{\Pamo}(\rho)$ is isomorphic to  $(a,b)\mapsto\parres_{\Pamo}(\rho)(b,a)$.
\end{lemma}
\begin{proof}
Conjugation by $u_1$ preserves $\Pamo$ and $\Pamo^+$ and gives rise to the automorphism $(a,b)\mapsto (b,a)$ of $(\GL(2,q)^2)^0\cong\Pamo/\Pamo^+$.
\end{proof}


\begin{proof}[Thm.\,\ref{thm:02:F_1_GSp4} iii)]By \eqref{eq:paraho_res_twist_Pamo}, we can assume without loss of generality that $\mu_0=1$.
The irreducible admissible representations $\rho$ of type I, IIIa, IIIb and VII are Klingen induced and the statement is implied by Prop.~\ref{Prop:paramod_res_Klingen_induced}. Representations of type IIa, IIb, X, XIa and XIb are Siegel induced and given by Prop.~\ref{Prop:paramod_res_Siegel_induced}.

For the trivial representation $\Ione_{\GSp(4,F)}$ of type IVd, the parahoric restriction at $\Pamo$ is clearly the trivial representation $[\Ione,\Ione]$.
The representation $\rho=L(\nu^{3/2}\St,\nu^{-3/2})$ of type IVc is the non-trivial constituent of the Klingen induced representation $\nu^2\times\nu^{-1}\Ione_{\GSp(2,F)}$ \cite[(2.9)]{Roberts-Schmidt}. By exactness and Prop.~\ref{Prop:paramod_res_Klingen_induced}, its parahoric restriction is $\parres_\Pamo(\rho)=[\St,\Ione]+[\Ione,\St]+[\Ione,\Ione]$. By the analogous argument with Prop.~\ref{Prop:paramod_res_Siegel_induced}, the parahoric restriction for representations of type IVa and IVb is clear. 

For $\rho=L(\nu^{1/2}\xi\St,\nu^{-1/2}\xi)$ of type Vc with an unramified quadratic character $\xi=\xi_u$, the parahoric restriction $\parres_{\Pamo}(\rho)$ is contained in $\parres_\Pamo(\nu^{1/2}\xi_u\Ione_{\GL(2)}\rtimes\nu^{-1/2})=[1\times1,1\times1]$. By \cite[Thm.~3.30]{My_PhD_Thesis}, $\parres_{\Pamo}(\rho)$ has a generic subquotient, which must be $[\St,\St]$. There is exactly one further constituent in $\parres_{\Pamo}(\rho)$, because the Klingen parahoric restriction $\parres_{\PahoKl}(\rho)$ contains two constituents by Table \ref{tab:02:non-special_par_res_GSp4}. By Lemma~\ref{Lem:paramod_res_symmetry}, this can only be $[\Ione,\Ione]$.
For tamely ramified quadratic character $\xi=\xi_t$, the parahoric restriction $\parres_{\Pamo}(\rho)$ is given by \cite[\S3.3.3]{My_PhD_Thesis}. For types Va,Vb and Vd the result is clear by Prop.~\ref{Prop:paramod_res_Siegel_induced} and exactness.

The representation $\rho=\tau(S,\nu^{-1/2})$ of type VIa is a constituent of the Klingen induced representation $\nu^{1/2}\St_{\GL(2,F)}\rtimes\nu^{-1/2}$ and of the Siegel induced representation $1\rtimes\St_{\GSp(2,F)}$ \cite[(2.11)]{Roberts-Schmidt}. By Prop.~\ref{Prop:paramod_res_Klingen_induced} and \ref{Prop:paramod_res_Siegel_induced}, the parahoric restriction at $\Pamo$ is a subquotient of $[\St,\St]+[\Ione,\St]+[\St,\Ione]$. The Klingen parahoric restriction $\parres_{\PahoKl}(\rho)$ has three irreducible constituents, so \eqref{eq:trans_parahori_res} implies $\parres_{\Pamo}(\rho)\cong [\St,\St]+[\Ione,\St]+[\St,\Ione]$. For types VIb, VIc and VId the result is clear by exactness.


Representations of type VIII are irreducible subquotients of $1\rtimes\cuspGL$ with $\cuspGL$ of depth zero. Their paramodular restriction is either $[\wt{\pi},\Ione]+[\Ione,\wt{\pi}]$ or $[\wt{\pi},\St]+[\St,\wt{\pi}]$ by Prop.~\ref{Prop:paramod_res_Klingen_induced} and Lemma \ref{Lem:paramod_res_symmetry}. The rest of the argument is analogous to the hyperspecial case:
By \cite[Cor.~4.23]{My_PhD_Thesis}, the character value of $$\parres_{\Pamo}(\tau(S,\cuspGL)-\tau(T,\cuspGL))$$ on the conjugacy class stably conjugate to $(\diag(\alpha\beta,\alpha^q\beta^q),\diag(\alpha\beta^q,\alpha^q\beta))$ for $\alpha,\beta\in\FF_{q^2}^\times$ with $\alpha,\beta,\alpha\beta,\alpha\beta^q\notin\FF_q^\times$ is given by
\begin{gather*}
-2(\Lambda(\alpha)+\Lambda(\alpha^q))(\Lambda(\beta)+\Lambda(\beta^q))\\=2(-\Lambda(\alpha\beta)-\Lambda^q(\alpha\beta))+2(-\Lambda(\alpha\beta^q)-\Lambda^q(\alpha\beta^q)).
\end{gather*}
This implies  $\parres_{\Pamo}(\tau(T,\cuspGL))=[\wt{\pi},\St]+[\St,\wt{\pi}]$ and $\parres_{\Pamo}(\tau(S,\cuspGL))=[\wt{\pi},\Ione]+[\Ione,\wt{\pi}]$.

For type IX, see \cite[\S3.3.2]{My_PhD_Thesis}.
\end{proof}

\section{Tables}\label{sec:Tables}
The irreducible admissible representations of $\GSp(4,F)$ have been classified by Sally and Tadi\'{c} \cite{Sally_Tadic}. We use the notation of Roberts and Schmidt \cite{Roberts-Schmidt}.

For $i=0,1,2$ let $\mu_i: F^{\times}\to\CC^{\times}$ be tamely ramified or unramified characters.
Let $\cuspGL$ be an arbitrary cuspidal irreducible admissible representation of $\GL(2,F)$ of depth zero. Its hyperspecial restriction $\parres_{\GL(2,\fo)}(\cuspGL)=\wt{\pi}=\wt{\pi}_\Lambda$ is an irreducible cuspidal representation of $\GL(2,q)$. Up to a sign it is the Deligne-Lusztig representation attached to a character $\Lambda$ of $\FF_{q^2}^\times$ in general position, i.e. $\Lambda\neq\Lambda^q$. 
The contragredient of $\wt{\pi}$ is denoted $\wt{\pi}^\vee$.

The non-trivial unramified quadratic character of $F^{\times}$ is denoted $\xi_u$. For odd $q$ let $\xi_t$ be one of the two tamely ramified quadratic characters which reduce to the non-trivial quadratic character $\quadcharFq$ of $\FF_q^\times$. For even $q$ there is no tamely ramified quadratic character.

\textit{Table~\ref{tab:02:tab_F1_on_GSp4}.}
For even $q$, irreducible characters of $\Sp(4,q)$ have been classified by Enomoto \cite{Enomoto1972}. By the isomorphism $\Sp(4,q)\times\GL(1,q)\cong \GSp(4,q)$, $(x,t)\mapsto t\cdot x$, the irreducible characters of $\GSp(4,q)$ can be classified in terms of their restriction to $\Sp(4,q)$ and their central character. Fix a generator $\hat{\theta}$ of the cyclic character group of $\FF_{q^2}^\times$ and denote its restrictions to $\FF_q^\times$ by $\hat{\gamma}$ and to $\FF_{q^2}^\times[q+1]$ by $\hat{\eta}$, respectively.
Let $k_i\in \ZZ/(q-1)\ZZ$ be such that $\hat{\gamma}^{k_i}=\wt{\mu_i}$.
Let $l\in\ZZ/(q^2-1)\ZZ$ be such that $\Lambda=\hat{\theta}^{l}$ and let $l'$ be the image of $l$ under the canonical projection $\ZZ/(q^2-1)\ZZ\twoheadrightarrow\ZZ/(q+1)\ZZ$ so that the restriction of $\Lambda$ to $\FF_{q^2}^\times[q+1]$ is $\hat{\eta}^{l'}$. If $(q+1)l=0$, there is a unique preimage $l''$ of $l$ under the canonical injection $\ZZ/(q+1)\ZZ\hookrightarrow\ZZ/(q^2-1)\ZZ$.

For odd $q$, irreducible characters of $\GSp(4,q)$ have been classified by Shinoda \cite{Shinoda}. A character $\Lambda$ of $\FF_{q^2}^\times$ with $\Lambda^{q+1}=1$ factors over a character $\omega_\Lambda$ of $\FF_{q^2}^\times[q+1]$ via $\Lambda(\alpha)=\omega_\Lambda(\alpha^{q-1})$. If $\Lambda^{q-1}$ is the quadratic character $\Lambda_0$ of $\FF_{q^2}^\times$, there is a unique character $\lambda'$ of $\FF_{q^2}^\times[2(q-1)]$ with $\Lambda(\alpha)=\lambda'(\alpha^{(q+1)/2})$.

\textit{Table~\ref{tab:02:non-special_par_res_GSp4}.} The trivial and the Steinberg representation of $\GL(2,q)$ are denoted $\Ione$ and $\St$, respectively. The parabolic induction of the character $\mu_1\boxtimes\mu_2$ of the standard torus is denoted $\mu_1\times\mu_2$. For typographical reasons, we write
\begin{gather*}
 A(\wt{\mu_1},\wt{\mu_2},\wt{\mu_0}) =\wt{\mu_1}\boxtimes\wt{\mu_2}\boxtimes\wt{\mu_0}+\wt{\mu_1}\boxtimes\wt{\mu_2}^{-1}\boxtimes\wt{\mu_2}\wt{\mu_0}+\wt{\mu_1}^{-1}\boxtimes\wt{\mu_2}\boxtimes\wt{\mu_1}\wt{\mu_0}\\+\wt{\mu_1}^{-1}\boxtimes\wt{\mu_2}^{-1}\boxtimes\wt{\mu_1}\wt{\mu_2}\wt{\mu_0},\\
B(\wt{\mu_1},\wt{\mu_2},\wt{\mu_0})=\wt{\mu_1}\boxtimes(\wt{\mu_2}\rtimes\wt{\mu_0})+\wt{\mu_1}^{-1}\boxtimes(\wt{\mu_2}\rtimes\wt{\mu_1}\wt{\mu_0}),\\
C(\wt{\mu_1},\wt{\mu_2},\wt{\mu_0})=(\wt{\mu_1}\times\wt{\mu_2})\boxtimes\wt{\mu_0}+(\wt{\mu_1}^{-1}\times\wt{\mu_2})\boxtimes\wt{\mu_1}\wt{\mu_0}.
\end{gather*}

\textit{Table~\ref{tab:02:tab_r_x_paramod}.} We call a representation of $(\GL(2,q)^2)^0=\{(a,b)\in \GL(2,q)^2\mid \det a=\det b\}$ \emph{generic} if it is generic with respect to the unipotent character 
\begin{equation*}
\left(\begin{pmatrix}1&x\\&1\end{pmatrix},\,\begin{pmatrix}1&y\\&1\end{pmatrix}\right)\mapsto\wt{\psi}(x+y).
\end{equation*} This does not depend on the choice of the non-trivial additive character $\wt{\psi}$ of $\FF_q$.
The irreducible representations of $\GL(2,q)\times\GL(2,q)$ are $\sigma_1\boxtimes\sigma_2$ for irreducible representations $\sigma_i$ of $\GL(2,q)$. We denote the restriction of $\sigma_1\boxtimes\sigma_2$ to $(\GL(2,q)^2)^0$ by $[\sigma_1,\sigma_2]$. This restriction is irreducible unless $\quadcharFq\sigma_1\cong\sigma_1$ and $\quadcharFq\sigma_2\cong\sigma_2$ for the non-trivial quadratic character $\quadcharFq$ of $\FF_q^\times$, then it splits into an equidimensional direct sum of a generic constituent $[\sigma_1,\sigma_2]_+$ and a non-generic constituent $[\sigma_1,\sigma_2]_-$. 
These are all the irreducible representations of $(\GL(2,q)^2)^0$. The twist of a representation $\sigma$ of $(\GL(2,q)^2)^0$ by a character $\wt{\mu}$ of $\FF_q^\times$ is defined by
\begin{equation*}
 \wt{\mu}[\sigma_1,\sigma_2]
=[(\wt{\mu}\circ\det)\otimes\sigma_1,\sigma_2]
=[\sigma_1,(\wt{\mu}\circ\det)\otimes\sigma_2].
\end{equation*}

\begin{sidewaystable}
\caption{Parahoric restriction at $\PahoHy$ for non-cuspidal irreducible admissible representations of $\GSp(4,F)$.\label{tab:02:tab_F1_on_GSp4}}
\begin{small}
\begin{tabular}{lcrrrr}
\toprule
type &     $\rho$ of $\GSp(4,F)$                &$\parres_\PahoHy(\rho)|_{\Sp(4,q)}$ (even $q$)   & $\parres_\PahoHy(\rho)$ (odd $q$) & central character  & $\dim\parres_\PahoHy(\rho)$      \\\midrule
  I &$\mu_1\times\mu_2\rtimes\mu_0$          &$\chi_1(k_1,k_2) $            &$X_1(\wt{\mu_1},\wt{\mu_2},\wt{\mu_0})$		   &$\wt{\mu_1}\wt{\mu_2}\wt{\mu_0}^2$&$(q+1)^2(q^2+1)$\\    
 IIa &$\mu_1\St\rtimes\mu_0$                 &$\chi_{10}(k_1)$              &$\chi_4(\wt{\mu_1},\wt{\mu_0})$  &$\wt{\mu_1}^2\wt{\mu_0}^2$ &$q(q+1)(q^2+1)$ \\    
 IIb &$\mu_1\Ione\rtimes\mu_0$               &$\chi_6(k_1)$                 &$\chi_3(\wt{\mu_1},\wt{\mu_0})$  &$\wt{\mu_1}^2\wt{\mu_0}^2$ &$(q+1)(q^2+1)$           \\    
 IIIa&$\mu_1\rtimes\mu_0\St$                 &$\chi_{11}(k_1)$              &$\chi_2(\wt{\mu_1},\wt{\mu_0})$  &$\wt{\mu_1}\wt{\mu_0}^2$   &$q(q+1)(q^2+1)$ \\   
 IIIb&$\mu_1\rtimes\mu_0 \Ione$              &$\chi_7(k_1)$                 &$\chi_1(\wt{\mu_1},\wt{\mu_0})$  &$\wt{\mu_1}\wt{\mu_0}^2$   &$(q+1)(q^2+1)$           \\    
 IVa&$\mu_0 \St_{\GSp(4,F)}$                 &$\theta_4$                 &$\theta_5(\wt{\mu_0})$           &$\wt{\mu_0}^2$             &$q^4$           \\    
 IVb&$L(\nu^2,\nu^{-1}\mu_0 \St)$            &$\theta_1+\theta_2$&$\theta_1(\wt{\mu_0})+\theta_3(\wt{\mu_0})$   &$\wt{\mu_0}^2$ &$q^3+q^2+q$      \\ 
 IVc&$L(\nu^{3/2} \St,\nu^{-3/2}\mu_0)$      &$\theta_1+\theta_3$&$\theta_1(\wt{\mu_0})+\theta_4(\wt{\mu_0})$   &$\wt{\mu_0}^2$ &$q^3+q^2+q$          \\ 
 IVd&$\mu_0\Ione_{\GSp(4,F)}$                &$\theta_0$                 &$\theta_0(\wt{\mu_0})$           &$\wt{\mu_0}^2$             &$1$                      \\ 
 Va &$\delta(\left[\xi_u,\nu\xi_u\right],\nu^{-1/2}\mu_0)$&$\theta_3+\theta_4$  &$\theta_4(\wt{\mu_0})+\theta_5(\wt{\mu_0})$ &$\wt{\mu_0}^2$ & $q^4+\frac{1}{2}q(q^2+1)$\\ 
    &$\delta(\left[\xi_t,\nu\xi_t\right],\nu^{-1/2}\mu_0)$&$-$              &$\tau_3(\wt{\mu_0})$             &$\wt{\mu_0}^2$             &$q^2(q^2+1)$        \\ 
 Vb &$L(\nu^{1/2}\xi_u \St,\nu^{-1/2}\mu_0)$ &$\theta_1$                 &$\theta_1(\wt{\mu_0})$           &$\wt{\mu_0}^2$             &$\frac{1}{2}q(q+1)^2$  \\
    &$L(\nu^{1/2}\xi_t \St,\nu^{-1/2}\mu_0)$ &$-$                           &$\tau_2(\wt{\mu_0})$             &$\wt{\mu_0}^2$             &$q(q^2+1)$          \\
 Vc &$L(\nu^{1/2}\xi_u \St,\nu^{-1/2}\xi_u\mu_0)$&$\theta_1$             &$\theta_1(\wt{\mu_0})$           &$\wt{\mu_0}^2$             &$\frac{1}{2}q(q+1)^2$  \\
    &$L(\nu^{1/2}\xi_t \St,\nu^{-1/2}\xi_t\mu_0)$&$-$                       &$\tau_2(\wt{\mu_0}\quadcharFq)$     &$\wt{\mu_0}^2$             &$q(q^2+1)$       \\
 Vd &$L(\nu\xi_u,\xi_u\rtimes\nu^{-1/2}\mu_0)$&$\theta_0+\theta_2$&$\theta_0(\wt{\mu_0})+\theta_3(\wt{\mu_0})$    &$\wt{\mu_0}^2$    &$1+\frac{1}{2}q(q^2+1)$ \\ 
    &$L(\nu\xi_t,\xi_t\rtimes\nu^{-1/2}\mu_0)$&$-$                          &$\tau_1(\wt{\mu_0})$             &$\wt{\mu_0}^2$             &$q^2+1$         \\ 
 VIa&$\tau(S,\nu^{-1/2}\mu_0)$               &$\theta_1+\theta_4$&$\theta_1(\wt{\mu_0})+\theta_5(\wt{\mu_0})$             &$\wt{\mu_0}^2$    &$q^4+\frac{1}{2}q(q+1)^2$\\ 
 VIb&$\tau(T,\nu^{-1/2}\mu_0)$               &$\theta_2$                 &$\theta_3(\wt{\mu_0})$       &$\wt{\mu_0}^2$             &$\frac{1}{2}q(q^2+1)$       \\ 
 VIc&$L(\nu^{1/2} \St,\nu^{-1/2}\mu_0)$      &$\theta_3$                 &$\theta_4(\wt{\mu_0})$       &$\wt{\mu_0}^2$             &$\frac{1}{2}q(q^2+1)$       \\ 
 VId&$L(\nu,1_{F^{\times}}\rtimes\nu^{-1/2}\mu_0)$&$\theta_0+\theta_1$ &$\theta_0(\wt{\mu_0})+\theta_1(\wt{\mu_0})$&$\wt{\mu_0}^2$ &$1+\frac{1}{2}q(q+1)^2$      \\
\midrule 
 VII&$\mu_1\rtimes\cuspGL$                     &$\chi_3(k_1,l')$         	  & $X_3(\Lambda,\wt{\mu_1})$ &$\wt{\mu_1}\cdot\Lambda|_{\FF_q^\times}$&$q^4-1$\\
 VIIIa&$\tau(S,\cuspGL)$			   &$\chi_{13}(l')$        	  &$\chi_8(\Lambda)$          &$\Lambda|_{\FF_q^\times}$      &$q(q-1)(q^2+1)$ \\ 
 VIIIb&$\tau(T,\cuspGL)$			   &$\chi_{9}(l')$		  &$\chi_7(\Lambda)$          &$\Lambda|_{\FF_q^\times}$      &$(q-1)(q^2+1)$  \\
  IXa& $\delta(\nu\xi_u,\nu^{-1/2}\cuspGL)$	   &$\chi_{13}(l')$               &$\chi_8(\Lambda)$          &$\Lambda|_{\FF_q^\times}$      &$q(q-1)(q^2+1)$ \\ 
     & $\delta(\nu\xi_t,\nu^{-1/2}\cuspGL)$    &$-$                           &$\tau_5(\lambda')$  &$\quadcharFq\cdot\Lambda|_{\FF_q^\times}$&$q^2(q^2-1)$ \\ 
  IXb& $L(\nu\xi_u,\nu^{-1/2}\cuspGL)$         &$\chi_{9}(l')$                &$\chi_7(\Lambda)$          &$\Lambda|_{\FF_q^\times}$      &$(q-1)(q^2+1)$  \\
     & $L(\nu\xi_t,\nu^{-1/2}\cuspGL)$         &$-$                           &$\tau_4(\lambda')$  &$\quadcharFq\cdot\Lambda|_{\FF_q^\times}$&$q^2-1$      \\
\midrule 
 X &$\cuspGL\rtimes\mu$		           &$\chi_{2}(l)$                 &$X_2(\Lambda,\wt{\mu_0})$    &$\wt{\mu_0}^2\cdot\Lambda|_{\FF_q^\times}$&$q^4-1$\\ 
 XIa&$\delta(\nu^{1/2}\cuspGL,\nu^{-1/2}\mu_0)$  &$\chi_{12}(l'')$              &$\chi_6(\omega_\Lambda,\wt{\mu_0})$&$\wt{\mu_0}^2$             &$q(q-1)(q^2+1)$ \\ 
 XIb&$L(\nu^{1/2}\cuspGL,\nu^{-1/2}\mu_0)$	   &$\chi_{8}(l'')$               &$\chi_5(\omega_\Lambda,\wt{\mu_0})$&$\wt{\mu_0}^2$             &$(q-1)(q^2+1)$  \\
\bottomrule
\end{tabular}
\end{small}
\end{sidewaystable}

\begin{sidewaystable}
\caption{Parahoric restriction at $\Iwahori, \PahoKl, \PahoSi$ for non-cuspidal irreducible admissible representations of $\GSp(4,F)$. \label{tab:02:non-special_par_res_GSp4}}
\begin{small}
\begin{tabular}{lcrrr}
\toprule
type &$\rho$ of $\GSp(4,F)$    & $\parres_{\Iwahori}(\rho)\in \Rep((\FF_q^\times)^{3})$ & $\parres_{\PahoKl}(\rho)\in \Rep(\FF_q^\times\times\GSp(2,q))$ & $\parres_{\PahoSi}(\rho)\in \Rep(\GL(2,q)\times\FF_q^\times)$\\\midrule
  I  &$\mu_1\times\mu_2\rtimes\mu_0$& $A(\wt{\mu_1},\wt{\mu_2},\wt{\mu_0})+A(\wt{\mu_2},\wt{\mu_1},\wt{\mu_0})$ & $B(\wt{\mu_1},\wt{\mu_2},\wt{\mu_0})+B(\wt{\mu_2},\wt{\mu_1},\wt{\mu_0})$ & 
$C(\wt{\mu_1},\wt{\mu_2},\wt{\mu_0})+C(\wt{\mu_1},\wt{\mu_2}^{-1},\wt{\mu_2}\wt{\mu_0})$\\
 IIa &$\mu_1\St\rtimes\mu_0$       &$A(\wt{\mu_1},\wt{\mu_1},\wt{\mu_0})$  & $B(\wt{\mu_1},\wt{\mu_1},\wt{\mu_0})$  & $\wt{\mu_1}\St\boxtimes\wt{\mu_0}+\wt{\mu_1}^{-1}\St\boxtimes\wt{\mu_0}\wt{\mu_1}^2$\\    
     &                             &                             &         &$+(\wt{\mu_1}\times\wt{\mu_1}^{-1})\boxtimes\wt{\mu_0}\wt{\mu_1}$\\
 IIb &$\mu_1\Ione\rtimes\mu_0$     &$A(\wt{\mu_1},\wt{\mu_1},\wt{\mu_0})$  &$B(\wt{\mu_1},\wt{\mu_1},\wt{\mu_0})$    & $\wt{\mu_1}\Ione\boxtimes\wt{\mu_0}+\wt{\mu_1}^{-1}\Ione\boxtimes\wt{\mu_0}\wt{\mu_1}^2$ \\
     &                             &                             &         &$+(\wt{\mu_1}\times\wt{\mu_1}^{-1})\boxtimes\wt{\mu_0}\wt{\mu_1}$\\
 IIIa&$\mu_1\rtimes\mu_0\St$       & $\wt{\mu_1}\boxtimes 1\boxtimes\wt{\mu_0}+\wt{\mu_1}^{-1}\boxtimes 1\boxtimes\wt{\mu_1}\wt{\mu_0}$    &$\wt{\mu_1}\boxtimes\wt{\mu_0}\St+1\boxtimes\wt{\mu_1}\rtimes\wt{\mu_0}$  &     $C(\wt{\mu_1},1,\wt{\mu_0})$ \\ 
     &                             & $+1\boxtimes\wt{\mu_1}\boxtimes\wt{\mu_0}+1\boxtimes\wt{\mu_1}^{-1}\boxtimes\wt{\mu_1}\wt{\mu_0}$  & $+\wt{\mu_1}^{-1}\boxtimes\wt{\mu_1}\wt{\mu_0}\St$&\\
 IIIb&$\mu_1\rtimes\mu_0 \Ione$    & $\wt{\mu_1}\boxtimes 1\boxtimes\wt{\mu_0}+\wt{\mu_1}^{-1}\boxtimes 1\boxtimes\wt{\mu_1}\wt{\mu_0}$    &$\wt{\mu_1}\boxtimes\wt{\mu_0}\Ione+1\boxtimes\wt{\mu_1}\rtimes\wt{\mu_0}$  & $C(\wt{\mu_1},1,\wt{\mu_0})$ \\ 
     &                             &    $+1\boxtimes\wt{\mu_1}\boxtimes\wt{\mu_0}+1\boxtimes\wt{\mu_1}^{-1}\boxtimes\wt{\mu_1}\wt{\mu_0}$  & $+\wt{\mu_1}^{-1}\boxtimes\wt{\mu_1}\wt{\mu_0}\Ione$&\\
 IVa&$\mu_0 \St_{\GSp(4)}$         & $1\boxtimes 1\boxtimes\wt{\mu_0}$               &$1\boxtimes\wt{\mu_0}\St$                                   &$\St\boxtimes\wt{\mu_0}$\\
 IVb&$L(\nu^2,\nu^{-1}\mu_0 \St)$  & $3(1\boxtimes 1\boxtimes\wt{\mu_0})$            &$1\boxtimes\wt{\mu_0}\Ione+2(1\boxtimes\wt{\mu_0}\St)$      &$\St\boxtimes\wt{\mu_0}+2(\Ione\boxtimes\wt{\mu_0})$\\ 
 IVc&$L(\nu^{3/2} \St,\nu^{-3/2}\mu_0)$& $3(1\boxtimes 1\boxtimes\wt{\mu_0})$       &$2(1\boxtimes\wt{\mu_0}\Ione)+1\boxtimes\wt{\mu_0}\St$      &$2(\St\boxtimes\wt{\mu_0})+\Ione\boxtimes\wt{\mu_0}$\\ 
 IVd&$\mu_0\Ione_{\GSp(4)}$        & $1\boxtimes 1\boxtimes\wt{\mu_0}$               &$1\boxtimes\wt{\mu_0}\Ione$                                 &$\Ione\boxtimes\wt{\mu_0}$\\
 Va &$\delta(\left[\xi,\nu\xi\right],\nu^{-1/2}\mu_0)$&$\wt{\xi}\boxtimes \wt{\xi}\boxtimes\wt{\mu_0}+\wt{\xi}\boxtimes \wt{\xi}\boxtimes\wt{\xi}\wt{\mu_0}$ & $\wt{\xi}\boxtimes(\wt{\xi}\rtimes\wt{\mu_0})$ &$\wt{\xi}\St\boxtimes\wt{\mu_0}+\wt{\xi}\St\boxtimes\wt{\xi}\wt{\mu_0}$ \\
 Vb &$L(\nu^{1/2}\xi\St,\nu^{-1/2}\mu_0)$       & $\wt{\xi}\boxtimes \wt{\xi}\boxtimes\wt{\mu_0}+\wt{\xi}\boxtimes \wt{\xi}\boxtimes\wt{\xi}\wt{\mu_0}$ &$\wt{\xi}\boxtimes(\wt{\xi}\rtimes\wt{\mu_0})$ &$\wt{\xi}\St\boxtimes\wt{\mu_0}+\wt{\xi}\Ione\boxtimes\wt{\xi}\wt{\mu_0}$ \\
 Vc &$L(\nu^{1/2}\xi\St,\nu^{-1/2}\xi\mu_0)$   & $\wt{\xi}\boxtimes \wt{\xi}\boxtimes\wt{\mu_0}+\wt{\xi}\boxtimes \wt{\xi}\boxtimes\wt{\xi}\wt{\mu_0}$ &$\wt{\xi}\boxtimes(\wt{\xi}\rtimes\wt{\mu_0})$ &$\wt{\xi}\Ione\boxtimes\wt{\mu_0}+\wt{\xi}\St\boxtimes\wt{\xi}\wt{\mu_0}$ \\
 Vd &$L(\nu\xi,\xi\rtimes\nu^{-1/2}\mu_0)$     & $\wt{\xi}\boxtimes \wt{\xi}\boxtimes\wt{\mu_0}+\wt{\xi}\boxtimes \wt{\xi}\boxtimes\wt{\xi}\wt{\mu_0}$ &$\wt{\xi}\boxtimes(\wt{\xi}\rtimes\wt{\mu_0})$ &$\wt{\xi}\Ione\boxtimes\wt{\mu_0}+\wt{\xi}\Ione\boxtimes\wt{\xi}\wt{\mu_0}$ \\
 VIa&$\tau(S,\nu^{-1/2}\mu_0)$                  & $3(1\boxtimes 1\boxtimes\wt{\mu_0})$  &$1\boxtimes\wt{\mu_0}\Ione+2(1\boxtimes\wt{\mu_0}\St)$ & $2(\St\boxtimes\wt{\mu_0})+\Ione\boxtimes\wt{\mu_0}$\\ 
 VIb&$\tau(T,\nu^{-1/2}\mu_0)$                  & $1\boxtimes 1\boxtimes\wt{\mu_0}$  &$1\boxtimes\wt{\mu_0}\St$  &$\Ione\boxtimes\wt{\mu_0}$   \\ 
 VIc&$L(\nu^{1/2} \St,\nu^{-1/2}\mu_0)$         & $1\boxtimes 1\boxtimes\wt{\mu_0}$  &$1\boxtimes\wt{\mu_0}\Ione$&$\St\boxtimes\wt{\mu_0}$        \\ 
 VId&$L(\nu,1_{F^{\times}}\rtimes\nu^{-1/2}\mu_0)$& $3(1\boxtimes 1\boxtimes\wt{\mu_0})$  &$2(1\boxtimes\wt{\mu_0}\Ione)+1\boxtimes\wt{\mu_0}\St$ &$\St\boxtimes\wt{\mu_0}+2(\Ione\boxtimes\wt{\mu_0})$    \\\midrule
 VII&$\mu_1\rtimes\cuspGL$                      &$0$     & $\wt{\mu_1}\boxtimes\wt{\pi}+\wt{\mu_1}^{-1}\boxtimes\wt{\mu_1}\wt{\pi}$     & $0$  \\
VIIIa&$\tau(S,\cuspGL)$				&$0$     & $1\boxtimes\wt{\pi}$              & $0$          \\ 
VIIIb&$\tau(T,\cuspGL)$				&$0$	 & $1\boxtimes\wt{\pi}$              & $0$          \\
 IXa& $\delta(\nu\xi,\nu^{-1/2}\cuspGL)$		&$0$	 & $\wt{\xi}\boxtimes\wt{\pi}$              & $0$          \\ 
 IXb& $L(\nu\xi,\nu^{-1/2}\cuspGL)$ 	        &$0$	 & $\wt{\xi}\boxtimes\wt{\pi}$              & $0$          \\\midrule
 X &$\cuspGL\rtimes\mu_0$			&$0$	 & $0$     &
 $\wt{\pi}\boxtimes\wt{\mu_0}+(\wt{\pi})^\vee\boxtimes\wt{\omega_{\pi}}\wt{\mu_0}$\\ 
 XIa&$\delta(\nu^{1/2}\cuspGL,\nu^{-1/2}\mu_0)$	&$0$	 & $0$                               & $\wt{\pi}\boxtimes\wt{\mu_0}$\\ 
 XIb&$L(\nu^{1/2}\cuspGL,\nu^{-1/2}\mu_0)$	      	&$0$	 & $0$                               & $\wt{\pi}\boxtimes\wt{\mu_0}$\\\bottomrule
\end{tabular}
\end{small}
\end{sidewaystable}

\begin{table}
\caption{Parahoric restriction at $\Pamo$ for non-cuspidal irreducible admissible representations of $\GSp(4,F)$. The index is determined by $\xi_t(\varpi)=\pm1$. \label{tab:02:tab_r_x_paramod}}
\centering
 \begin{small}
\begin{tabular}{lcrr}
\toprule
type &$\rho$ of $\GSp(4,F)$                 & $\parres_{\Pamo}(\rho)\in\Rep((\GL(2,q)^2)^0)$                      & $\dim\parres_{\Pamo}(\rho)$ \\\midrule
   I &$\mu_1\times\mu_2\rtimes\mu_0$        & $\wt{\mu_0}[1\times\wt{\mu_1},1\times\wt{\mu_2}]+\wt{\mu_0}[1\times\wt{\mu_2},1\times\wt{\mu_1}]$&$2(q+1)^2$\\ 
  IIa&$\mu_1\St\rtimes\mu_0$                & $\wt{\mu_0}[1\times\wt{\mu_1},1\times\wt{\mu_1}]$                   &$(q+1)^2$   \\
  IIb&$\mu_1\Ione\rtimes\mu_0$              & $\wt{\mu_0}[1\times\wt{\mu_1},1\times\wt{\mu_1}]$                   &$(q+1)^2$   \\
 IIIa&$\mu_1\rtimes\mu_0\St$                & $\wt{\mu_0}[1\times\wt{\mu_1},\St]+\wt{\mu_0}[\St,1\times\wt{\mu_1}]$         &$2q(q+1)$   \\
 IIIb&$\mu_1\rtimes\mu_0 \Ione$             & $\wt{\mu_0}[1\times\wt{\mu_1},\Ione]+\wt{\mu_0}[\Ione,1\times\wt{\mu_1}]$     &$2(q+1)$    \\
  IVa&$\mu_0 \St_{\GSp(4,F)}$               & $\wt{\mu_0}[\St,\St]$                                              &$q^2$       \\
  IVb&$L(\nu^2,\nu^{-1}\mu_0 \St)$          & $\wt{\mu_0}[\St,\St]+\wt{\mu_0}[\Ione,\St]+\wt{\mu_0}[\St,\Ione]$               &$q^2+2q$    \\
  IVc&$L(\nu^{3/2} \St,\nu^{-3/2}\mu_0)$    & $\wt{\mu_0}[\Ione,\Ione]+\wt{\mu_0}[\Ione,\St]+\wt{\mu_0}[\St,\Ione]$           &$2q+1$      \\
  IVd&$\mu_0\Ione_{\GSp(4,F)}$              & $\wt{\mu_0}[\Ione,\Ione]$                                          &$1$         \\
   Va&$\delta(\left[\xi_u,\nu\xi_u\right],\nu^{-1/2}\mu_0)$&$\wt{\mu_0}[\Ione,\St]+\wt{\mu_0}[\St,\Ione]$        &$2q$        \\
     &$\delta(\left[\xi_t,\nu\xi_t\right],\nu^{-1/2}\mu_0)$&$\wt{\mu_0}[1\times\quadcharFq,1\times\quadcharFq]_\pm$           &$(q+1)^2/2$\\
%
%
   Vb&$L(\nu^{1/2}\xi_u \St,\nu^{-1/2}\mu_0)$ &$\wt{\mu_0}[\Ione,\Ione]+\wt{\mu_0}[\St,\St]$                     &$q^2+1$     \\
     &$L(\nu^{1/2}\xi_t\St,\nu^{-1/2}\mu_0)$  &$\wt{\mu_0}[1\times\quadcharFq,1\times\quadcharFq]_\mp$                       &$(q+1)^2/2$ \\
   Vc&$L(\nu^{1/2}\xi_u \St,\nu^{-1/2}\xi_u\mu_0)$& $\wt{\mu_0}[\Ione,\Ione]+\wt{\mu_0}[\St,\St]$                &$q^2+1$     \\
     &$L(\nu^{1/2}\xi_t\St,\nu^{-1/2}\xi_t\mu_0)$&$\wt{\mu_0}[1\times\quadcharFq,1\times\quadcharFq]_\mp$                     &$(q+1)^2/2$ \\
   Vd&$L(\nu\xi_u,\xi_u\rtimes\nu^{-1/2}\mu_0)$ &$\wt{\mu_0}[\Ione,\St]+\wt{\mu_0}[\St,\Ione]$                   &$2q$        \\
     &$L(\nu\xi_t,\xi_t\rtimes\nu^{-1/2}\mu_0)$&$\wt{\mu_0}[1\times\quadcharFq,1\times\quadcharFq]_\pm$                       &$(q+1)^2/2$ \\
  VIa&$\tau(S,\nu^{-1/2}\mu_0)$                &$\wt{\mu_0}[\St,\St]+\wt{\mu_0}[\St,\Ione]+\wt{\mu_0}[\Ione,\St]$               &$q^2+2q$    \\ 
  VIb&$\tau(T,\nu^{-1/2}\mu_0)$                &$\wt{\mu_0}[\St,\St]$                                                       &$q^2$       \\ 
  VIc&$L(\nu^{1/2} \St,\nu^{-1/2}\mu_0)$          &$\wt{\mu_0}[\Ione,\Ione]$                                                   &$1$         \\ 
  VId&$L(\nu,1_{F^{\times}}\rtimes\nu^{-1/2}\mu_0)$&$\wt{\mu_0}[\Ione,\Ione]+\wt{\mu_0}[\St,\Ione]+\wt{\mu_0}[\Ione,\St]$       &$2q+1$      \\\midrule
%
 VII &$\mu_1\rtimes\cuspGL$                        &$[1\times\wt{\mu_1},\wt{\pi}]+[\wt{\pi},1\times\wt{\mu_1}]$               &$2(q^2-1)$  \\
VIIIa&$\tau(S,\cuspGL)$			           &$[\Ione,\wt{\pi}]+[\wt{\pi},\Ione]$                                       &$2(q-1)$    \\ 
VIIIb&$\tau(T,\cuspGL)$			           &$[\St,\wt{\pi}]+[\wt{\pi},\St]$                                           &$2(q-1)q$   \\
  IXa&$\delta(\nu\xi_u,\nu^{-1/2}\cuspGL)$	   &$[\St,\wt{\pi}]+[\wt{\pi},\St]$                                           &$2(q-1)q$   \\
     &$\delta(\nu\xi_t,\nu^{-1/2}\cuspGL)$         &$[\wt{\pi},1\times\quadcharFq]_\mp+[1\times\quadcharFq,\wt{\pi}]_\mp$       &$q^2-1$    \\
  IXb&$L(\nu\xi_u,\nu^{-1/2}\cuspGL)$              &$[\Ione,\wt{\pi}]+[\wt{\pi},\Ione]$                                       &$2(q-1)$    \\

     &$L(\nu\xi_t,\nu^{-1/2}\cuspGL)$              &$[\wt{\pi},1\times\quadcharFq]_{\pm}+[1\times\quadcharFq,\wt{\pi}]_\pm$       &$q^2-1$    \\\midrule
%
   X &$\cuspGL\rtimes\mu_0$		           &$0$                                                     &$0$         \\ 
  XIa&$\delta(\nu^{1/2}\cuspGL,\nu^{-1/2}\mu_0)$   &$0$                                                     &$0$         \\ 
  XIb&$L(\nu^{1/2}\cuspGL,\nu^{-1/2}\mu_0)$	   &$0$                                                     &$0$ \\\bottomrule
\end{tabular}
\end{small}
\end{table}

\textit{Acknowledgements.}
The author wants to express his gratitude to R. Weissauer and U. Weselmann for valuable discussions.

\bibliographystyle{spmpsci}
\bibliography{Main_Parahori_res_GSp4}

\begin{thebibliography}{10}
\providecommand{\url}[1]{{#1}}
\providecommand{\urlprefix}{URL }
\expandafter\ifx\csname urlstyle\endcsname\relax
  \providecommand{\doi}[1]{DOI~\discretionary{}{}{}#1}\else
  \providecommand{\doi}{DOI~\discretionary{}{}{}\begingroup
  \urlstyle{rm}\Url}\fi

\bibitem{Borel_Iwahori_Invariants}
Borel, A.: {Admissible Representations of a Semi-Simple Group over a Local
  Field with Vectors Fixed under an Iwahori Subgroup}.
\newblock Inv. Math. \textbf{35}, 233--259 (1976)

\bibitem{Breeding}
Breeding-Allison, J.: {Irreducible Characters of {$GSp(4,\mathbb{F}_q)$}}.
\newblock Ramanujan J. \textbf{36}(3), 305--354 (2015)

\bibitem{BushnellHenniart}
Bushnell, C.J., Henniart, G.: {The Local Langlands Conjecture for {$GL(2)$}}.
\newblock {Grund\-lehren der Mathematischen Wissenschaften}. Springer (2006)

\bibitem{Enomoto1972}
Enomoto, H.: {The characters of the finite symplectic group {$Sp(4,q)$},
  {$q=2^f$}}.
\newblock Osaka J. Math. \textbf{9}, 75--94 (1972)

\bibitem{Moy-Prasad1996}
Moy, A., Prasad, G.: {Jacquet functors and unrefined minimal {K}-types}.
\newblock Comment. Math. Helvetici \textbf{71}, 98--121 (1996)

\bibitem{Roberts-Schmidt}
Roberts, B., Schmidt, R.: {Local Newforms for {$GSp(4)$}}, \emph{{Lecture Notes
  in Mathematics}}, vol. 1918, 1 edn.
\newblock Springer (2007)

\bibitem{My_PhD_Thesis}
R{\"o}sner, M.: {Parahoric restriction for GSp(4) and the inner cohomology of
  Siegel modular threefolds}.
\newblock Ph.D. thesis, University of Heidelberg (2016).
\newblock {doi:10.11588/heidok.00021401}

\bibitem{Sally_Tadic}
Sally, P., Tadi{\'c}, M.: {Induced Representations and classifications for
  {$GSp(2,F)$} and {$Sp(2,F)$}}.
\newblock M{\'e}moires Soci{\'e}t{\'e} Math{\'e}matique de France \textbf{52},
  75--133 (1994)

\bibitem{Shinoda}
Shinoda, K.I.: {The characters of the finite conformal symplectic group
  $CSp(4,q)$}.
\newblock Communications in Algebra \textbf{10}(13), 1369--1419 (1982)

\bibitem{Ta91}
Tadi{\'c}, M.: {On {Jacquet} Modules of Induced representations of p-adic
  symplectic groups}.
\newblock In: {Harmonic Analysis on Reductive Groups}, \emph{{Progress in
  Mathematics}}, vol. 101, pp. 305--314. Birkh{\"a}user (1991)

\bibitem{Vigneras_Representations_Modulaire}
Vign{\'e}ras, M.F.: {Repr{\'e}sentations {$l$}-modulaires d{\rq}un groupe
  r{\'e}ductif {$p$}-adique avec {$l\neq p$}}, \emph{{Progress in
  Mathematics}}, vol. 137.
\newblock Birkh{\"a}user (1996)

\bibitem{Vigneras_Barcelona}
Vign{\'e}ras, M.F.: {Irreducible Modular Representations of a reductive
  {$p$}-adic group and simple Modules for {Hecke} Algebras}.
\newblock In: C.~Casacuberta, et~al. (eds.) {European Congress of Mathematics,
  Barcelona}, \emph{{Progress in Mathematics}}, vol. 201, pp. 117--133.
  Birkh{\"a}user (2001)

\bibitem{Vig_Schur}
Vign{\'e}ras, M.F.: {Schur Algebras of reductive {$p$}-adic groups}.
\newblock Duke Math. J. \textbf{116}(1), 35--75 (2003)

\bibitem{Weissauer200903}
Weissauer, R.: {Endoscopy for {$GSp(4)$} and the Cohomology of Siegel Modular
  Threefolds}, \emph{{Lecture notes in Mathematics}}, vol. 1968.
\newblock Springer (2009)

\end{thebibliography}

\end{document}